\documentclass[graybox]{svmult}

\usepackage{mathptmx}       
\usepackage{helvet}         
\usepackage{courier}        
\usepackage{type1cm}        
\usepackage{makeidx}         
\usepackage{graphicx}        
\usepackage{multicol}        
\usepackage[bottom]{footmisc}

\usepackage{amsmath}
\usepackage{amssymb}
\usepackage{amscd}
\usepackage{indentfirst}

\newcommand{\Bsp}{{\boldsymbol B}}     
\newcommand{\BspN}{(\Bsp, \, \|\ebbes\|_\Bsp)}     
\newcommand{\ebbes}{\mbox{$\,\cdot\,$}}     
\newcommand{\COPRdN}{{\big( \COPsp(\Rst^d), \, \|\ebbes\|_\COPsp \big)}}     
\newcommand{\COPsp}{{\Csp'_{\negthinspace 0}}}     
\newcommand{\Rst}{{\mathbb R}}     
\newcommand{\Csp}{{\boldsymbol C}}     
\newcommand{\CORd}{{\COsp(\Rst^d)}}     
\newcommand{\COsp}{{\Csp_{\negthinspace 0}}}     
\newcommand{\CORdN}{{\big( \COsp(\Rst^d), \, \|\ebbes\|_\infty \big)}}     
\newcommand{\CbRd}{{\Cbsp(\Rdst)}}     
\newcommand{\Cbsp}{{\Csp_{\negthinspace b}}}     
\newcommand{\Rdst}{{{\Rst^d}}}     
\newcommand{\CbRdN}{{\big( \Cbsp(\Rst^d), \, \|\ebbes\|_\infty \big)}}     
\newcommand{\CcG}{{\Ccsp(G)}}     
\newcommand{\Ccsp}{{\Csp_{\negthinspace c}}}     
\newcommand{\CcRd}{{\Ccsp(\Rst^d)}}     
\newcommand{\Cst}{{\mathbb C}}     
\newcommand{\CuRd}{{\Cusp(\Rdst)}}     
\newcommand{\Cusp}{{\Csp_{\negthinspace\text{\it ub}}}}     
\newcommand{\CubRd}{{\Cusp(\Rdst)}}     
\newcommand{\CubRdN}{{\big( \CuRd, \, \|\ebbes\|_\infty \big)}}     
\newcommand{\Drho}{{\operatorname{D}_{\rho}}}     
\newcommand\FLi{{\mathcal F}{\negthinspace \Lisp}}     
\newcommand{\Lisp}{{\Lsp^1}}     
\newcommand{\FLiRd}{{ \FLi(\Rdst) }}     
\newcommand\FLiRdN{\big( \FLiRd, \, \|\ebbes\|_{\FLisp} \big)}     
\newcommand{\FLisp}{{ {\mathcal F}\Lisp}}     
\newcommand{\FLinsp}{{\mathcal F}\Linsp}     
\newcommand{\Linsp}{{\Lsp^\infty}}     
\newcommand{\FT}{{\operatorname{{\mathcal F} \negthinspace}}}     
\newcommand{\cG}{\mathscr{G}}     
\newcommand\Hilb{\mathcal H}     
\newcommand{\IFT}{\operatorname{\mathcal F}^{-1}}     
\newcommand{\LiG}{{\Lisp(G)}}     
\newcommand{\LiRd}{{\Lisp \negthinspace (\Rst^d)}}     
\newcommand{\LiRdN}{\big( \LiRd, \, \|\ebbes\|_1 \big)}     
\newcommand{\LicG}{{\Lisp(\cG)}}     
\newcommand\Liloc{ {\Lsp^1_{ \negthinspace \it{loc}} } }     
\newcommand{\Lsp}{{\boldsymbol L}}     
\newcommand{\LinG}{{\Linsp(G)}}     
\newcommand{\LinRd}{{\Linsp(\Rst^d)}}     
\newcommand{\LinRdN}{\big( \LinRd, \, \|\ebbes\|_\infty \big)}     
\newcommand\LinfRd{ L^{\infty}(\Rst^{d})}     
\newcommand{\LinfRdN}{\big( \LinfRd, \, \|\ebbes\|_\infty \big)}     
\newcommand{\LpRd}{{\Lpsp(\Rst^d)}}     
\newcommand{\Lpsp}{{\Lsp^p}}     
\newcommand{\LpRdN}{\big( \LpRd, \, \|\ebbes\|_p \big)}     
\newcommand{\LqRd}{{ \Lsp^q(\Rdst) }}
\newcommand{\LqRdN}{\big( \LqRd, \, \|\ebbes\|_q \big)}     
\newcommand{\Ltsp}{{\Lsp^2}}     
\newcommand{\LtRd}{{\Ltsp(\Rst^d)}}     
\newcommand{\LtRdN}{\big( \LtRd, \, \|\ebbes\|_2 \big)}     
\newcommand{\LtRtd}{{\Ltsp(\Rst^{2d})}}     
\newcommand\Ltzo{{(\Ltsp([0,1])}}     
\newcommand\LtzoN{{(\Ltsp([0,1]),\| \ebbes\|_2)}}     
\newcommand{\MbRd}{{\Mbsp(\Rst^d)}}     
\newcommand{\Mbsp}{{\Msp_{\negthinspace b}}}     
\newcommand\MbRdN{{(\Mbsp(\Rst^d), \| \ebbes \|_\Mbsp )}}     
\newcommand{\Msp}{{\boldsymbol M}}     
\newcommand{\MdRd}{ {\Mdsp (\Rdst)}}     
\newcommand{\Mdsp}{{\Msp_{\negthinspace d}}}     
\newcommand{\MiRdN}{{ (\Msp^1(\Rdst), \| \ebbes \|_{\Msp^1}) }}     
\newcommand{\Modul}{\operatorname{M}}
\newcommand{\Nst}{{\mathbb N}}     
\newcommand\QsRd{{\boldsymbol Q}_s (\Rdst) }     
\newcommand\QsRdN{\big( \QsRd, \, \|\ebbes\|_{\Qssp} \big)}     
\newcommand\Qssp{{\Qsp_s}}     
\newcommand\Qsp{ \boldsymbol Q }     
\newcommand{\Qst}{{\mathbb Q}}     
\newcommand{\Rdsth}{{\widehat{\Rst}^d}}     
\newcommand{\Rtdst}{{\Rst^{2d}}}     
\newcommand\Rtst{ {{\Rst}^{2} }}     
\newcommand{\SOGTrRd}{{ (\SOsp,\Ltsp,\SOPsp)(\Rdst) }}     
\newcommand{\SOsp}{{\Ssp_{\negthinspace 0}}}     
\newcommand{\SOPsp}{{\Ssp_{\negthinspace 0}'}}     
\newcommand\SOLSOP{ (\SOsp,\Ltsp,\SOPsp)}     
\newcommand{\SOPG}{{\SOPsp(G)}}     
\newcommand{\SOPRd}{{\SOPsp(\Rst^d)}}     
\newcommand{\SOPRdN}{(\SOPRd , \| \ebbes \|_{\SOPsp} ) }     
\newcommand{\Ssp}{{\boldsymbol S}}     
\newcommand{\SORd}{{\SOsp(\Rst^d)}}     
\newcommand{\SORdN}{\big( \SORd, \|\ebbes\|_\SOsp \big)}     
\newcommand{\ScPRd}{{\ScPsp(\Rst^d)}}     
\newcommand{\ScPsp}{{\Scsp'}}     
\newcommand{\Scsp}{{\boldsymbol{\mathcal S}}}     
\newcommand{\ScRd}{{\Scsp(\Rst^d)}}     
\newcommand\SegN{{ (\Ssp, \| \ebbes \|_\Ssp)}}     
\newcommand{\Strho}{{\operatorname{St}_\rho}}     
\newcommand\TFRd{{{\Rdst\,{\times}\,\Rdsth}}}     
\newcommand{\TFd}{{{ \Rdst \times \Rdsth }}}     
\newcommand{\Trans}{\operatorname{T}}
\newcommand{\Tst}{\mathbb{T}}     
\newcommand{\Ust}{\mathbb{U}}     
\newcommand\Vgsi{{V_{g _0}\sigma}}     
\newcommand\WFLili{{\Wsp(\FLi,\lisp)}}     
\newcommand{\Wsp}{{\boldsymbol W}}     
\newcommand{\lisp}{{\lsp^1}}     
\newcommand\WFLiliRd{{\Wsp(\FLi,\lisp)(\Rdst)}}     
\newcommand\WFLilin{{\Wsp(\FLi,\linsp)}}     
\newcommand{\linsp}{{\lsp^\infty}}     
\newcommand\WFLilinRd{{\Wsp(\FLi,\linsp)(\Rdst)}}     
\newcommand\WFLinli{ \Wsp( \FLinsp, \lisp)}     
\newcommand\WFLinlin{ \Wsp( \FLinsp, \linsp)}     
\newcommand{\WRd}{{\Wsp(\Rdst)}}     
\newcommand\WWRd{{ \WRd \cap \FT \WRd}}     
\newcommand{\Zdst}{{\Zst^d}}     
\newcommand{\Zst}{{\mathbb Z}}     
\newcommand\abZZd{ {a \Zdst \times b \Zdst}}     
\newcommand\bfone{{\bf 1}}     
\newcommand\cdto{\cdot} 
\newcommand{\checkm}{{^\checkmark}}     
\newcommand{\conjug}[1]{\overline{#1}}     
\newcommand{\cosp}{\operatorname{cosp}}     
\newcommand{\df}[1]{{\em #1}}
\newcommand{\eg}{e.g. \,}     
\newcommand{\eps}{\varepsilon}     
\newcommand{\epso}{{ \varepsilon > 0 }}     
\newcommand\fSON{{\|f\|_\SOsp}}     
\newcommand{\fSORd}{ f \in \SORd }     
\newcommand\fafSORd{{ \forall f \in \SORd}}     
\newcommand{\fhat}{{\widehat{f}}}     
\newcommand\fninf{{\|f\|_\infty}}     
\newcommand\gLam{ \left ( g_{\lambda} \right )_{\laL}}     
\newcommand\laL{\lambda \in \Lambda}     
\newcommand{\gd}{{\widetilde{g}}} 
\newcommand{\ghat}{{\widehat{g}}}     
\newcommand\gSON{{\|g\|_\SOsp}}

\newcommand\gin{{\|g\|_1}}     
\newcommand\glam{ g_{\lambda}}     
\newcommand\hatf{{\widehat{f}}}     
\newcommand\hatsi{{\widehat{\sigma}}}     
\newcommand{\ie}{i.e.}     
\newcommand{\intR}{\int_{\Rst}}     
\newcommand{\intRd}{\int_{\Rst^d}}     
\newcommand\intTFd{{\int_{\TFd}}}     
\newcommand\ipinf{{1 \leq p \leq \infty}}     
\newcommand{\lainLa}{{\lambda{\in}\Lambda}}     
\newcommand\liLam{{\lisp(\Lambda)}}     
\newcommand\liN{{\lisp(\Nst)}}     
\newcommand\limninf{{\lim_{n \to \infty}}}     
\newcommand{\linorm}[1]{{\lVert #1 \rVert_1}}     
\newcommand{\lsp}{{\boldsymbol\ell}}     
\newcommand\ltLam{{ \ltsp(\Lambda)}}     
\newcommand{\ltsp}{{\lsp^2}}     
\newcommand{\ltnorm}[1]{{\lVert #1 \rVert_2}}     
\newcommand{\ninN}{{{n{\in}\Nst}}}     
\newcommand\qandq{{ \quad \mbox{and} \quad }}     
\newcommand\siSOP{{ \sigma \in \SOPsp }}     
\newcommand\siSOPN{{\|\sigma\|_\SOPsp}}     
\newcommand\sigmal{ \sigma_\alpha }     
\newcommand{\sonorm}[1]{{\lVert #1 \rVert_\SOsp}}     
\newcommand{\sopnorm}[1]{{\lVert #1 \rVert_\SOPsp}}     
\newcommand{\spec}{\operatorname{spec}}     
\newcommand{\sumlaLa}{\sum_{\lambda\in\Lambda}}     
\newcommand{\sumnZ}{\sum_{n\in\Zst}}     
\newcommand\sumninf{{\sum_{n=-\infty}^{\infty}}}     
\newcommand\supngi{{\sup_{n \geq 1}}}     
\newcommand{\supnorm}[1]{{\lVert #1 \rVert_\infty}}     
\newcommand{\supp}{\operatorname{supp}}     
\newcommand\suprRd{{\sup_{r \in \Rdst}}}     
\newcommand\suth{{ \, | \, } }     
\newcommand{\tensor}{\otimes}     
\newcommand\veps{{\varepsilon}}     
\newcommand{\wdash}{{ w^* \negthinspace \mbox{-}}}     
\newcommand\wst{ w^{*} }     
\newcommand\wstlim{{ \wdash \lim \, }}     
\newcommand{\wwst}{{\wdash \wdash}}     

\def\cG{G}
\def\NSP#1{{(#1,\| \ebbes \|_{#1})}}

\def\seqg#1{(#1_n)_{n \geq 1}}

\makeindex             

\begin{document}
\title*{A Sequential Approach to Mild Distributions}
\titlerunning{Hans G. Feichtinger: A Sequential Approach to Mild Distributions}
\author{Hans G.~Feichtinger}
\authorrunning{Hans  G.~Feichtinger }
\institute{Hans G.~Feichtinger \at Faculty of Mathematics, University of Vienna  \\
\email{hans.feichtinger@univie.ac.at} }

\maketitle



\section{Abstract}

The Banach Gelfand Triple $\SOGTrRd$ consists of $\SORdN$, a very
specific  {\it Segal algebra} as  algebra of test functions, the Hilbert space
$\LtRdN$ and the dual space $\SOPRd$, whose elements are also called {\it ``mild
distributions''}. Together they provide a   universal tool for Fourier Analysis in its many
manifestations.

It is indispensable for a proper formulation of {\it Gabor Analysis},
but also useful for a distributional description of the classical (generalized)
Fourier transform (with Plancherel's Theorem and the Fourier Inversion Theorem as
core statements) or the foundations of
Abstract Harmonic Analysis, as it is not difficult to formulate this theory in the
context of LCA groups.
A new approach presented recently allows to introduce $\SORdN$ and hence $\SOPRdN$,
the space of ``mild distributions'', without Lebesgue integrals or tempered distributions.

The present note describes a more elementary, sequential approach,
inspired by the Lighthill approach to tempered distributions.
By drawing analogies to the real number system (of infinite
decimals) we hope to make it accessible to engineers.

The main topic of this article is thus an outline of the sequential
approach in this concrete setting and the clarification of the fact that
it is just another way of describing the Banach Gelfand Triple.
The objects of the extended domain for the Short-Time Fourier
Transform are (equivalence classes of) so-called {\it mild
Cauchy sequences} (in short {\bf ECmiCS}). Representatives are
sequences of bounded, continuous functions, which
correspond in a natural way to mild distributions as introduced
in earlier papers via duality theory.
Our key result shows how standard functional analytic arguments
combined with concrete properties of the Segal algebra $\SORdN$
can be used to establish this natural identification.

\section{Introduction}

It is the purpose of this article  to present the vector space of
``mild distributions'' as the natural object of all ``signals which
have a bounded spectrogram'' (or STFT: Short-Time Fourier Transform).
We will do this in an elementary way, inspired by the classical
sequential approach to distributions, in the spirit of Lighthill (\cite{li58})
or Bracewell (\cite{br83}, Chap.5), which is also popular among engineers.

The basic principle will be the idea of {\it completion}, which
allows to enlarge an object by adding elements which make it
``more complete''. This is a well-known concept familiar from
the introduction of the real number system: The field $\Qst$ of
rational numbers is great when it comes to multiplication and even
addition can be realized in a purely algebraic way, but unfortunately
it shows some incompleteness, since obviously there is no rational
number $q$ such that $q^2 = 2$.

In order to enlarge $\Qst$ and create a complete field, to be named
the real number system $\Rst$, one has several options, including
Dedekind's section, or the very concrete idea of introducing the
{\it infinite decimal expressions}, endowed with appropriate
computational rules, which turns $\Rst$ into a field which is
complete with respect to the Euclidean distance, i.e.\ where
every Cauchy sequence is convergent. Just recall:
\begin{definition} \label{CSdef1}
A sequence $(q_n)_{n \geq 1}$ in $\Qst$ is called a {\it Cauchy-sequence} whenever
relative distances of elements of the sequence tend to zero whenever
the labels of two elements are both large enough. In symbols:
$$ \forall \epso \, \, \exists n_0 \in \Nst \,\,\mbox{such that} \,\,
  m,n \geq n_0  \Rightarrow  |q_n - q_m| < \eps.
  $$
\end{definition}
Note that it is of course enough to allow rational numbers $\eps \in \Qst$
or only the choice $\eps = 1/n$ for any $n \in \Nst$. 

It is a standard fact that the field $\Qst$ is not complete (with
respect to the metric $d(q,q') = |q-q'|$, because one can find
a Cauchy sequence with $\limninf q_n^2 = 2$, but without finding
a limit $q_0 \in \Qst$ (which then would be a rational number with $q_0^2=2$,
which easily leads to a contradiction).

So instead of looking at $\Qst$ itself one adds (as a replacement
to the non-existent limits) the collection of Cauchy-sequences
representing the same (not yet existent) limit, in order to get
a complete metric space, i.e.\ a larger object that contains
$\Qst$ as a dense subspace, and such that within that larger object
every Cauchy sequence has a limit. Since many different
Cauchy-sequences ``represent the same real number'' one has
to choose either a fixed representation (known as the representation
of real numbers as infinite decimals), or one introduces equivalence
classes of rational numbers, based on the following equivalence
relation:
\begin{definition}
Two Cauchy-sequences $(q_n)_{n\geq 1}$ and $(p_k)_{k \geq 1}$ are
called {\it equivalent} if one can mix them in an arbitrary fashion and
still have  a Cauchy sequence, or in other words, if for $\epso$
there exists $n_1 \in \Nst$ such that
$$  n,k \geq n_1 \Rightarrow  |q_n - p_k| < \eps. $$
\end{definition}
One then goes on and shows that this is indeed an equivalence relation
on the vector space of Cauchy sequences (with natural element-wise operations),
and that one can transfer all the structure (like
addition, multiplication, inverse elements, absolute part) to this
larger object, just called $\Rst$.

It is also important to note that every element   $q \in \Qst$ defines
 an equivalence class, via the constant
sequence:  $q_n = q, \forall n \geq 1$. One has positive
elements in $\Rst$ and the absolute value is again well defined.
Moreover, using the standard diagonalization trick one can show
that every Cauchy-sequence in this new object (it is a rather
tricky object seen in this way) has in fact a limit within $\Rst$,
and hence  every positive $r \in \Rst$ has roots
of a (unique positive) square root, i.e.\ symbols like $\sqrt{2}$
make sense. In fact, this achievement of a (more or less unique)
completion of $(\Qst,| \ebbes|)$ is a  beautiful cornerstone
of mathematical analysis.

Of course this is nothing but a wordy description
(one of various possible) of the field of real numbers $\Rst$,
with the same operations as in $\Qst$, just by natural extension.


If one takes a somewhat wider perspective one realizes that
the decimal system is a very special system. One might use other
number systems, or continued fractions, but as a striking
insight (usually explained in  real analysis courses) any
of the completions obtained gives the same object,
resp.\ provides just an alternative description of $\Rst$, the
natural completion of $\Qst$ (with Euclidean metric).

The most abstract version of the underlying abstract principle
allows to embed any {\it metric space} isometrically into
a complete metric space, as a dense subspace. In other words, one can
create a situation, which is perfectly analogue to the situation of the
embedding $\Qst \hookrightarrow\Rst$: Clearly rational numbers, by definition
expressed as quotients of the form $p/q$, with $p,q \in \Nst, q \neq 0$,
have to be identified with decimal expressions, in the usual way, as we
have learnt in school. Once we have understood that $3/4$ is the same
as $0.75$ we do not care whether $ 3/4 \cdot 6/5 = 18/20 = 9/10$ has
been computed within the rational numbers or as $ 0.75 \cdot 1.2 = 0.75 + 0.15 =
0.9$. Moreover, $\Qst$ is dense in $\Rst$, because obviously
for any $\epso$ there exists some $k \in\Nst$ such that $(1/10)^k <
\varepsilon$ and hence the finite decimal expression for $x \in \Rst$
which keeps the first $k$ digits of $x$ will provide an $\varepsilon$-approximation
to $x$.

Generally, the method of completion consists of several steps:
\begin{enumerate} \item
First one considers {\it all possible Cauchy-sequences} (for a
given metric);
\item Then one forms {\it equivalence classes} of Cauchy-sequences,
which are the objects of the new space;
\item The {\it distance of the new elements} can be defined
   as the limit of the distances between elements of the CS generating these
  equivalence classes (independent from
   the representatives of the classes);
\item Any element of the {\it original space defines a constant
sequence} which is obviously a CS. The claimed natural embedding
mapping assigns simply every element in the given space the
corresponding constant equivalence class;
\item Then one verifies that this {\it natural embedding} of the
original metric space into the new one {\it is isometric}, so that
henceforth the copy with the new object arising from the
original objects can be identified with the elements of the
original space\footnote{For the case of $\Qst \hookrightarrow \Rst$
we recognize the rational numbers among all the infinite decimal expressions
as those which are periodic, for a suitably chosen period.}.
\item If further (linear) structure is available, e.g.\ if we start from a  normed
space or a normed algebra, the resulting complete object is then
a Banach space or a Banach algebra, and most operators defined
on the original space can be naturally extended to the new setting. 
\end{enumerate}

\section{The Short-Time Fourier Transform STFT}

First we introduce the  {\it Short-Time Fourier Transform}
(STFT), or {\it Sliding Window Fourier Transform}  for continuous functions
$f$ and compactly supported {\it window} $g$, using simply Riemann integrals.
In the traditional approach to {\it Gabor Analysis} (see \cite{gr01})
one starts by assuming that both the signal  $f$ to be analyzed
and the window $g$ used for localization are in $\LtRd$,
 but from a practical viewpoint the
non-symmetric assumption appears to be more natural, e.g.\ well
localized windows allowing for the analysis of much more general
signals to be analyzed.

We start with $\CbRd$ as the signal space    , the set of
bounded, continuous, (real or ) complex-valued functions on $\Rdst$.
Since we do not require any decay or summability or decay conditions
on   $f$   we assume further that
the {\it window function}  $g \in \CcRd$, is a continuous (real-valued) function with
compact support (i.e.\ for some $R>0$ one has $g(x) = 0$ for $|x| \geq R$),
and is thus Riemann integrable, with
$$ \|g\|_1 := \intRd |g(x)|dx < \infty.$$
\begin{definition} \label{STFTCbRd}
Given $f \in \CbRd$ and  $g \in \CcRd$ one
defines the STFT of the signal/function $f$ with respect to the
window $g$ as the following function, whose argument are points
of the $2d$-dimensional phase space, namely pairs $(t,s)$, with
\footnote{With $t$ from the so-called time-domain
and $s$ from the frequency domain!} $t,s \in \Rdst$
\begin{equation} \label{defVgf1}
V_g(f)(t,s) := \intRd f(x) g(x-t) e^{- 2 \pi i s \cdot x} dx.
\end{equation}
\end{definition}
For the case $d=1$ the absolute value $|V_gf(t,s)|$
represents the frequency content of an
audio-signal at time $t$ and at frequency $s$, comparable
to a graphical score. A quite realistic idea of what it is
can be obtained by watching the web-page
{\tt www.gaborator.com}, where you can even upload your
own piece of music (in WAV-format).

It is easy to show that the STFT of a bounded and continuous
function is a bounded and continuous function of two variables:
\begin{lemma} \label{STFTbdcont}
Given $f \in \CbRd$ and $g \in \CcRd$. Then
$V_g(f) \in \Cbsp(\TFd)$ and
\begin{equation} \label{supVgest}
 \|V_g(f)\|_\infty \leq \|g\|_1 \|f \|_\infty.
\end{equation}
In other words, for any fixed $g \in \CcRd$ 
the mapping
$f \mapsto V_g(f)$ is a continuous embedding from
$\NSP \CbRd$ into $\NSP {\Cbsp(\Rtdst)}$.
\end{lemma}
\begin{proof}
The obvious estimate of $V_g(f)(t,s)$ comes from
$$
|V_g(f)(t,s)| \leq \intRd |f(x-t)| |g(x)|dx \leq \fninf \gin.
$$
The (uniform) continuity in the time direction comes from
$$ |V_g(f)(t',s) - V_g(f)(t',s)| \leq \fninf \|g - T_{t-t'}g\|_1 < \eps$$
whenever $|t-t'| < \delta$, due to the uniform continuity of $g$.
Finally (also uniform) continuity in the frequency direction
comes in: Given $\epso$ one can find some $\delta > 0$ such that
$|s-s'| < \delta$ implies
$$  |e^{2 \pi i s \cdot x} - e^{2 \pi i s' \cdot x}| = |1 - e^{2 \pi i (s-s') \cdot
x}| < \eps $$
whenever $x \in \supp(g)$, which in turn implies
$$ |V_g(f)(t,s) - V_g(f)(t,s')|  \leq \fninf \gin \cdot \eps. $$
\end{proof}

It is our goal to show that in some sense the space $\SOPRd$ of
mild distributions on $\Rdst$ can be viewed as a {\it kind of completion}
of $\CbRd$. Alternatively it can be viewed as the ``largest natural
vector space of signals'' which have a bounded spectrogram.

In the rest of this paper we will show how to extend the domain
of the STFT to a larger space of signals, which we call ``mild distributions'',
making use of (mild) Cauchy sequences, but also
by verifying that this approach provides the user with just an
alternative approach to the Banach space $\SOPRdN$, which has
been introduced long ago and which has been used in a series
of papers over many years. It also constitutes the outer layer
of the so-called Banach Gelfand Triple $\SOLSOP$.

\section{The Usual Approach to $\SORdN$}

For mathematicians familiar with the theory of tempered
distributions as it is taught in many courses one can describe
the Banach Gelfand Triple $\SOGTrRd$, consisting of the
Segal algebra $\SORdN$, the Hilbert space $\LtRdN$ and the
dual space $\SOPRdN$ in the following way, using the
{\it extended STFT}, well defined for any $\sigma \in \ScPRd$
(consistent with the classical definition) via
\newpage
\begin{definition} \label{STFscpdef}
Given any non-zero (real-valued) $g \in \ScRd$ we define for $\siSOP$
$$  V_g(\sigma)(t,s) = \sigma( M_{-s} T_t g), \quad t,s \in \Rdst.$$
\end{definition}
Here we use the standard notations familiar from time-frequency
analysis, namely $[T_z g](x) = g(x-z)$ and $[M_s h](x) = e^{2 \pi i s \cdot x}
h(x)$,
where $s \cdot x = \langle s, x \rangle = \sum_{k=1}^d s_k x_k $ is the usual
scalar product in $\Rdst$. $M_s T_t$ is called a time-frequency shift operator.

Basic facts fir  the triple $\SOGTrRd$
read as follows (see \cite{fe81-2,gr01,cofelu08,ja18}):
\begin{theorem}
For fixed $ 0 \neq g \in \ScRd$ one has:
\begin{itemize}
  \item For $f \in \LtRd$ one has $V_g(f) \in \LtRtd$ and
   $$  \|V_g(f)\|_\LtRtd = \ltnorm{g} \ltnorm{f}
   \quad  f,g \in \LtRd.$$
  \item A function $f \in \LtRd$ belongs to $\SORd$ by
  definition, if $V_g(f) \in \Lisp(\Rtdst)$, and
  $$  \sonorm{f} :=  \intTFd |V_g(f)(t,s)| dt ds. $$
  Any function $f \in \SORd$ is continuous, bounded
  and absolutely Riemann-integrable. Thus $\SORdN \hookrightarrow \LqRdN, $
  for $1 \leq q \leq \infty$ and $\SORd \subset \CORd$. \vspace{2mm}
  \item A tempered distribution $\sigma \in \ScPRd$ belongs
   to $\SOPRd$ if and only if
  $$ \|V_g(\sigma)\|_\infty =\sup_{(t,s) \in \TFd}|V_g(\sigma)(t,s)| < \infty.$$
   Moreover, this expression defines an equivalent norm on $\SOPRdN$.
\end{itemize}
\end{theorem}
The (continuous) embedding of $\LpRdN$ (for any $p \in [1,\infty]$) into $\SOPRdN$
is realized via
\begin{equation} \label{embCbRd1}
 h \mapsto \sigma_h(f): = \intRd  f(x) h(x) dx, \quad \fSORd.
\end{equation}

Aside from the norm convergence in $\SOPRdN$ (resp.\ uniform convergence of
the corresponding STFTs $V_g(\sigma_n)$ it is important to discuss the so-called
$\wst$-convergence of sequences\footnote{Due to the separability of $\SORdN$ it
is enough to use sequences, while in general one should use {\it nets} for
a description of the $\wst$-topology!}:
\begin{definition} \label{wstsequ1}
A sequence
$\seqg{\sigma}$ is $\wst$-convergent in $\SOPRd$ to $\sigma_0$ if and only if
\begin{equation}\label{wstconvseq1}
  \limninf \sigma_n(f) = \sigma_0(f), \quad \fafSORd.
\end{equation}
\end{definition}
A linear mapping between two subspaces of $\SOPRd$ is $\wwst$continuous
if $\wst$-convergent sequences are mapped into $\wst$-convergent sequences.

As  illustration let us give a few examples as they
appear   in (engineering) books on Fourier Analysis, involving
typically some ``hand-waving'' argument:
\newpage
\begin{enumerate}
  \item Dirac sequences, obtained by compression of an integrable function;
  $$ \wstlim_{\rho \to 0}  \, \Strho g = \left (\intR g(x)dx\right ) \delta_0 = \hat{g}(0) \delta_0,
  $$
  where $[\Strho g](x) = \rho^{-d} g(x/\rho)$, usually with $\hat{g}(0)= 1$.
  \item Riemannian sums converging to the integral, e.g.\ for $f \in \SORd$:
  $$ \lim_{\alpha \to 0} \langle  \alpha^d \sum_{k \in \Zdst} {\delta_{\alpha
  k}}, f \rangle =
   \lim_{\alpha \to 0}  \, \alpha^d \sum_{k \in \Zdst} f(\alpha k)   = \intRd
   f(x)dx =
  \langle {\bf 1}, f\rangle.$$
  \item  For any $f \in \LiRd$ the periodic versions of that  function converge
      to the    original function (this is often used to motivate the form of the
       specific form of the Fourier integral for non-periodic functions):
  $$ \wstlim_{p \to \infty}  \sum_{k \in \Zdst}  T_{pk} f =  f. $$
\end{enumerate}

\section{Mild Cauchy Sequences} 

We will take the space $\CbRd$ as a starting point, a vector space
of ``decent signals'', where the usual vector space operations (addition,
linear combinations etc.) are well defined. It is a normed
vector space with respect to ordinary addition and scalar multiplication
of continuous functions. The expression
\begin{equation} \label{supnormdef}
 \|f\|_\infty := \suprRd |f(r)|
\end{equation}
defines the natural norm on this space, turning $\CbRdN$ into a
Banach space, in fact, it is even a Banach algebra with respect to
pointwise multiplication, since
\begin{equation}\label{supsubmult}
\supnorm{f\cdot g} \leq \supnorm{f} \cdot \supnorm{g},\quad f,g \in\CbRd.
\end{equation}
Occasionally we will need $\CORdN$, the closed ideal of function in $\CbRd$
which vanish at infinity, i.e.\ satisfy
\begin{equation}\label{COdef001}
  \lim_{|x| \to \infty} |f(x)| = 0.
\end{equation}
The Riesz Representation Theorem justifies to simply identify the dual space
of $\CORdN$ with the space $\MbRdN$ (of bounded, regular Borel measures).
The continuity of such a functional in the form
$$ \mu:  f \mapsto  \intRd f(x) d\mu(x), $$
 corresponds to the usual $\sigma$-additivity used in measure theory, and the
 functional norm is the same as the total variation (see \cite{fe16}).

Since we are interested in extending the concept of an STFT in an
{\it elementary way} from $\CbRd$ to some larger space
( yet to be defined\footnote{The extension of the
 field  $\Qst$ of rational numbers to the
field $\Rst$ of real numbers is not only enlarging the set of objects which
can be considered, but it also preserves to the extent possible all the
properties and axioms known to be valid for $\Qst$,
like addition, multiplication, inversion and so on, which are designed
in a way compatible with the original structure.})
we first look at a variant of the concept of a Cauchy sequence which
appears to be appropriate in the current context.
\begin{definition} \label{defmildCauchy}
Fix any non-zero $g \in \CcRd$ with $\hat{g} \in \LiRd$. 
A sequence $(h_n)_{n\geq 1}$ in $\CbRd$ is called a {\it mild Cauchy sequence}
if 
the sequence $\|V_g(h_n)\|_\infty$ is bounded and
\begin{equation}\label{mildCSdef}
  \left( V_g(h_n)(s,t) \right)_{n \geq 1}
\end{equation}
is a Cauchy-sequence (with respect to $n$) for every pair $(t,s) \in \TFd$.
\end{definition}
Since the field $\Cst$ of complex numbers is itself complete this is of
course equivalent to the assumption that the limit of such a sequence
exists, i.e., for each $(t,s) \in \TFd$ one has a pointwise limit
\begin{equation}\label{cmildCSlim}
  \exists \,  H(s,t) = \limninf V_g(h_n)(s,t)
\end{equation}
and of course $H$ is then a bounded function with
\begin{equation}\label{cmildLIM}
  \supnorm{H} \leq \sup_{n \geq 1} \|V_g(h_n)\|_\infty.
\end{equation}

\begin{remark}
One can show that the definition of a mild Cauchy sequence
does {\it not depend} on the a particular choice of $g$.
 In fact, even a slightly stronger claim is true:
 if condition (\ref{mildCSdef})
is valid for one nonzero function $g_1 \in \SORd$ it is also true
for any other function $g_2 \in \SORd$. This follows from
the atomic characterization of $\SORdN$, a kind of {\it exchange principle},
(see \cite{fe81-2,ja18}.
\end{remark}

We define {\it equivalence of Cauchy-sequences}
 in a rather obvious way:
\begin{definition}  \label{mildCSequ}
Two mild Cauchy-sequences $(h_n^{(1)})$ and $(h_k^{(2)})$
are called {\it mildly equivalent}  if they have the same limit, i.e.\
\begin{equation}
 \limninf V_g(h_n^{(1)})(s,t) =   \lim_{k \to \infty} V_g(h_k^{(2)})(s,t),
 \quad \forall (t,s) \in \TFd.
\end{equation}
\end{definition}

We also define a norm on the vector space of equivalence classes of
mild Cauchy-sequences (short: ECmiCS):
\begin{definition} \label{CSnorm2}
For a given ECmiCS ${\bf F}$  we define its norm:
\begin{equation} \label{CSnorm}
\|{\bf F}\|_{CS} = \inf \{ \sup_{n \geq 1} \|V_g(h_n)\|_\infty \},
\end{equation}
where the infimum is taken over all representatives of the
equivalence class ${\bf F}$.
\end{definition}
It is then not difficult (but lengthly) to verify that
this is actually a norm on the vector space of equivalence
classes, and that $\CbRdN$ is continuously embedded into
this space, see Cor. \ref{completeness}. 
We will observe later that $\CbRd$ is {\it not dense}
in the new space with respect to this norm, but only
in the natural topology (corresponding to the
$\wst$-convergence in $\SOPRd$).
What is less obvious (but a valid claim) is the
{\it completeness of this new space}: every mild Cauchy
sequence has a limit.

In order to establish appropriate terminology for discussion
and later reference let us introduce the acronym {\bf ECmiCS}
for an {\it {\bf E}quivalence {\bf C}lass
if {\bf mi}ld {\bf C}auchy {\bf S}equences}. These ECmiCS constitute
the new (enlarged) vector space  of objects,  in fact a normed
space with respect to the $CS$-norm. For the rest of this
note it will be convenient to use the symbol ${\bf F}$ for
such an equivalence class, hence $ \|{\bf F}\|_{CS}$ describes
the norm of an ECmiCS.

Using this terminology our observations reduce to the following
statement:
\begin{lemma} \label{CbECmiCS1}
 The space $\CbRdN$ is continuously embedded into the space
 of $ECmiCS$, endowed with the $CS$-norm.
\end{lemma}
One only has to check (and this is quite easy) that if $h \in \CbRd$
represents the zero ECmiCS that it has to be the zero-function, but
this is quite clear because it implies that
$ \intRd f(x) h(x)dx = 0$ for any $f \in \SORd$,   based on the
atomic characterization of $\SORdN$ (see (\ref{atomdec1})).

What is perhaps more interesting is the fact that one can extend
more or less all the usual manipulations to this {\it enlarged
space}, by just applying it to the individual mild Cauchy sequences,
and by verifying that they are compatible with the introduced
equivalence relation.

Above all we have translation, modulation, dilations, which can
be easily transferred to the extended space. The Fourier transform
is a bit more tricky. 

The decisive formula is (3.10) in \cite{gr01}, p.40, also called
{\it fundamental formula of time-frequency analysis} there.
\begin{equation}  \label{Vgfour1}
V_{\ghat}(\fhat)(\omega, -t) = e^{2 \pi i t \cdot \omega} \, V_g(f)(t,\omega),
\quad t,\omega \in \Rdst.
\end{equation}
For the standard choice $g = g_0$ (Gaussian window) this implies that the STFT of
$\hatf$ is (up to some harmless phase factors) just the (absolute value of the)
original STFT of $f$, rotated by $90$ degrees in phase space,
thus implying the isometric invariance of $\SORdN$ under the
Fourier transform.
Since at this point we do not have a Fourier transform
for $h \in \CbRd$  we cannot   use of this formula yet.


\section{The Functional Analytic Viewpoint}

In this section we will demonstrate that the approach chosen
in the section above is in fact equivalent to the one used in
various publications on the subject so far, or the approach
to the subject possible in the context of tempered distributions.
Of course, it is necessary to make use of result from
standard linear {\it functional analysis} in order to carry
out these identifications in a mathematically correct way.

Let us first recall in some more detail the characterization of
$\SOPRdN$ (as e.g.\ introduced in \cite{cofelu08}) as a
subspace of tempered distributions $\ScPRd$:
\begin{lemma}
The space $\SOPRd$ coincides with the space of all tempered
distributions with a bounded STFT (with respect to any
non-zero Schwartz function $g$):
$$ V_g(\sigma)(t,s) = \sigma(M_{-s} T_t g), \quad t,s \in \Rdst.$$
\end{lemma}
\begin{proof}
First of all let us remind of the fact that $\ScRd$ is dense in
$\SORdN$. Hence $\SOPRd$ is a subspace of $\ScPRd$ (continuously embedded).
A tempered distribution  $\sigma$ defines an element of $\SOPRd$ if
and only if it is bounded on $\ScRd$, but with respect to the norm
of $\SORdN$.

If one has such a bounded linear functional  one has
for any $g \in \SORd$
\begin{equation} \label{unifestimSOP}
\sup_{(t,s) \in \TFd} |V_g(\sigma)(t,s)|
   \leq \sup_{(t,s) \in \TFd}   \siSOPN  \sonorm{M_{-s} T_t g } =
     \siSOPN \gSON.
\end{equation}
Conversely assume that  
$\sigma \in \ScPRd$ is given with bounded STFT (for some
non-zero Schwartz function $g$).
Then one has, thanks to the {\it atomic representation}
of a generic element  $f \in \SORdN$ as
\begin{equation} \label{atomdec1}
 f = \sumninf c_n M_{-s_n} T_{t_n} g \quad \mbox{with} \,\,\,
 \sumninf |c_n| \leq C \fSON,
\end{equation}
for some constant $C > 0$ the following estimate for any $f \in \ScRd
\subset \SORd$:
$$  |\sigma(f)| \leq \sumninf |c_n| \sigma( M_{-s_n} T_{t_n} g  ) $$
and further the boundedness of $\sigma$ on $\SORdN$ via
$$  |\sigma(f)| \leq \sumninf |c_n| |V_g(\sigma)(t_n,s_n)|
 \leq C \fSON \|V_g(\sigma)\|_\infty.
 $$
Altogether we have established the equivalence of
$\| V_g(\sigma)\|_\infty $ and the standard dual norm on $\SOPRdN$,
given by $ \siSOPN := \sup_{\fSON \leq 1} |\sigma(f)|$.
\end{proof}

Since the following result will immediately provide a number of properties
 of the our ``completion'' of $\CbRd$ we want to give it next, in order
 to avoid elementary, but cumbersome arguments (in the spirit of
 sequential approaches to generalized functions, as promoted by
 Lighthill \cite{li58}).

\newpage
This is one of our main results:
\begin{theorem} \label{Cauchydual}
There is a natural identification of $\SOPRdN$ with ECmiCS,  the
normed space of equivalence classes of mild Cauchy sequences in $\CbRd$,
with equivalence of norms.
\end{theorem}
\begin{proof}
Given any equivalence class of mild Cauchy-sequence ${\bf F}$  and
$\epso$ let us choose as {\it mild CS} $(f_n)_{n \geq 1}$ in $\CbRd$ with
$$ \supngi  \|V_g(f_n)\|_\infty \leq  \|{\bf F}\|_{CS} + \eps. $$

Then the sequence $(f_n)_{n \geq 1}$ can be viewed as a bounded
sequence $(\sigma_n)_{n \geq 1}$ in $\SOPRdN$, via
$$ \sigma_n(f) = \intRd f(x) f_n(x) dx, \quad f \in \SORd.$$
In fact,
$$ \supngi \|\sigma_n\|_\SOPsp \leq  \|{\bf F}\|_{CS} + \eps
  \leq  2 \|{\bf F}\|_{CS}.$$

By the atomic decomposition formula (\ref{atomdec1})
this sequence is uniformly bounded in $\SOPRd$. Therefore
it is enough to show that $(\sigma_n(f))_{n \geq 1}$
is a Cauchy-sequence for any $f \in \ScRd$ (the existence
of the limit follows then automatically). 
Given $\epso$ and $f \in \SORd$ let us choose a finite
sum of the form $ h = \sum_{n=1}^K c_n  M_{s_n} T_{t_n} g $
\begin{equation}
  \sigma_r(M_{s_n}T_{t_n} g) = V_g (\sigma_r)(t_n,s_n).
\end{equation}
with
$$\|f - h\|_\SOsp < \eps/(10 \|{\bf F}\|_{CS}),$$
 and of course
$\sum_{n \geq 1} |c_n|  \leq C \|f\|_\SOsp.$

Given the finite sequence $(t_n,s_n)_{n=1}^K$ in $\TFd$
the pointwise Cauchy-condition implies that we  can
find an index $n_1 \in \Nst$ such that one has for $r_1,r_2 \geq n_1$
\begin{equation}  \label{finconv1}
|V_g(f_{r_1})(t_n,s_n) - V_g(f_{r_2})(t_n,s_n)| < \eps \cdot (2C \fSON)^{-1},
\quad\mbox{for} 1 \leq n \leq K.
\end{equation}
Consequently one has
\begin{equation}
 |\sigma_{r_1}(h) - \sigma_{r_2}(h)| \leq
\sum_{n=1}^K |c_n|  |V_g(f_{r_1})(t_n,s_n) - V_g(f_{r_2})(t_n,s_n)|
\leq  \eps/2.
\end{equation}
This implies for  $r_1,r_2 \geq n_1$
\begin{equation}
|\sigma_{r_1}(f) - \sigma_{r_2}(f)| \leq
|\sigma_{r_1}(f-h)| +  \eps/2  +  |\sigma_{r_2}(h-f)| \leq
\end{equation}
\begin{equation}
\leq  2 \sup_{n \geq 1} \|\sigma_n\|_\SOsp \|f - h\|_\SOsp + \eps/2
\end{equation}
\begin{equation}  \leq   2/5  \eps + \eps/2 < \eps. \end{equation}

This shows that there exists some $\sigma$ with
$V_g(\sigma)$ being the pointwise limit of the
mild distributions $(\sigma_n)$ and thus
$$ \|\sigma\|_\SOPsp   = \sup |V_g(\sigma)(t,s)| \leq \|F\|_{CS} + \eps,$$
but since this claim is valid for any $\epso$ we have
$$ \siSOPN \leq \|{\bf F}\|_{CS}. $$

It remains to mention (details are left to the interested reader)
that the assignment $F \mapsto \sigma$ is in fact well defined,
i.e.\ it is in fact a mapping from the {\it equivalence}
class ${\bf F}$, not depending on the representative used. But this
follows from the uniqueness of the STFT: Given $\sigma_1,\sigma_2
\in \SOPRd$ with $V_g(\sigma_1) = V_g(\sigma_2)$ one has of course
$\sigma_1 = \sigma_2$.
\end{proof}
Since any dual of a normed space is a complete, normed space,
i.e.\ a Banach space, we have as an immediate consequence of the
isomorphism just stated:
\begin{corollary} \label{completeness}
The space of equivalence classes of mild Cauchy sequences
in $\CbRd$ with its natural norm 
is a Banach space. The embedding of $\CbRdN$ into
this space via constant sequences, i.e.\ by $f \mapsto
(f_n)$, with $f_n = f$ for all $n \geq 1$, is continuous,
with dense range (in the $\wst$-sense).
\end{corollary}
\begin{proof}
Obviously the completeness of $\SOPRdN$ transfers to the
``mild completion'' of $\CbRd$. A direct proof would be
possible, by elementary means, but it would be a bit cumbersome (and
less informative). The density follows from the $\wst$-density
of $\SORd$ in $\SOPRd$, combined with the inclusion
$$ \SORd \hookrightarrow \CbRd \hookrightarrow \SOPRd. $$
\end{proof}
\begin{remark}
Of course it would be possible to provide a direct proof, by means
of absolutely convergent series, but this argument would much
longer, but of course more elementary in terms of tools.
\end{remark}

Now we have to take care of the converse: Every element $\siSOP$
defines an equivalence class of distributions, in such a way that
this assignment is the inverse to the embedding just discussed.

For this purpose we will use the regularization properties which
arise from the convolution relations and pointwise multiplication
of Wiener amalgam spaces, which provide the following facts (see
e.g.\cite{fe83,fe81-2,fegr85}):
 \begin{proposition} 
The following inclusions are valid for the Wiener amalgam spaces
relevant for the treatment of $\SORd = \WFLiliRd$ and its dual
space:
\begin{enumerate}
  \item $\WFLili(\Rdst) \ast \WFLinlin(\Rdst) \subset \WFLilin(\Rdst)$;
  \item $\WFLilin(\Rdst) \cdot \WFLili(\Rdst) \subset \WFLili(\Rdst)$;
  \item $\WFLili(\Rdst) \cdot \WFLinlin(\Rdst) \subset \WFLinli(\Rdst)$.
\end{enumerate}
In fact, the Wiener amalgam space $\WFLilinRd$ coincides
(with equivalence of norms) with the space of pointwise
multipliers of $\SORdN$.
\end{proposition}
Using the fact which has been the original definition (in \cite{fe81-2},
see also \cite{gr01} where these spaces are introduced as modulation spaces)  that
$$ \SORd = \WFLili(\Rdst) \qandq \SOPRd = \WFLinlin(\Rdst)$$
with equivalence of norms one easily combines these relations to
obtain
\begin{equation} \label{SOSOPREG1}
 [\SOPRd \ast \SORd] \cdot \SORd \subset \SORd \qandq
  [\SOPRd \cdot \SORd] \ast \SORd,
\end{equation}
of course again combined with corresponding norm estimates, stating
that there exists some $C'> 0$ (depending on the choice of the norms)
such as (for example)
\begin{equation} \label{SOSOPREG2}
\|(\sigma \ast g) \cdot h\|_\SOsp \leq  C' \siSOPN \|g\|_\SOsp \|h\|_\SOsp,
\quad \forall \siSOP \, \,\, \mbox{and} \,\,
\,\forall g,h \in \SOsp.
\end{equation}

In order to create a {\it mild Cauchy sequence} representing
$\siSOP$ it is thus enough to make use of bounded approximate
units from in $\SORd$, i.e.\ an Dirac sequence forming
a bounded approximate unit for convolution in $\LiRdN$ (i.e.
bounded with respect to the $\LiRd$-norm) consisting entirely
of elements from the dense subspace $\SORd \subset \LiRdN$,
and another bounded approximate unit for pointwise multiplication,
now bounded with respect to the Fourier algebra $\FLiRdN$.

Of course one can get one from the other. In particular,
it is clear, that for a bounded sequence $(g_n)_{n \geq 1}$
in $\LiRdN$ with
\begin{equation} \label{LiRdAU1}
\| g_n \ast g - g \|_\Lisp \to 0 \,\, \mbox{for} \,\,
  n \to \infty, \,\, \forall g \in \LiRd,
\end{equation}
then (by the convolution theorem) $ h_n  = \widehat{g_n}$ is
a bounded sequence in $\FLiRdN$ with
\begin{equation} \label{FLiRdAU1}
\| h_n \ast h - h \|_\FLisp \to 0 \quad \mbox{for} \, \,\,
  n \to \infty, \,\, \forall h \in \FLiRd,
\end{equation}
and in fact vice versa.

The most convenient way to produce such a Dirac sequence is of
course (area preserving) compression of a function $g \in \SORd$ with
$ \intRd g(x)dx = \hat{g}(0) = 1$,
using the dilation operator
$\Strho$ given by
$$ [\Strho (g)](z) = \rho^{-d} g(z/\rho), \quad \rho > 0, z \in \Rdst,$$
satisfying (\ref{LiRdAU1}) for $\rho_n \to 0$,
with $$\|\Strho(g)\|_\Lisp = \|g\|_\Lisp,\quad g \in \LiRd, \rho > 0.$$
The Fourier transform of such a sequence is characterized by the
``value-preserving'' dilation operator
$\Drho$, given by
$$ [\Drho(h)](z) = h(\rho z), \quad \rho > 0, z \in \Rdst,$$
with $$\| \Drho(h) \|_\infty = \|h\|_\infty, \quad h \in \CbRd, \rho > 0.$$
In other words, we have the {\it intertwining relation}
\begin{equation} \label{interwdil}
\FT(\Strho(g)) = \Drho(\FT(g)), \quad g \in \LiRd, \, \rho > 0.
\end{equation}

Whenever $\intRd g(x) dx = 1 $ for $g \in \LiRd$ or
equivalently $h(0) = \ghat(0) = 1$ for $ h \in \FLiRd$
the corresponding families
$$ (\Strho_{\negthinspace n}(g))_{n \geq 1} \qandq
(\Drho_{\negthinspace n}(h))_{n \geq 1}$$
form such approximate units for any (fixed) sequence $\rho_{\negthinspace n} \to 0$.

For our purpose it is most convenient to choose $g = g_0$, given as
 $g_0(t) = e^{- \pi \|t\|^2}  = \prod_{j=1}^d  e^{- \pi t_j^2},  $
the normalized Gauss function, which belongs to $\ScRd \subset \SORd$
and is Fourier invariant, i.e.\ we have $g = g_0 = \widehat{g_0} = h$
and our assumptions are   satisfied.

Putting all these observations together we get: For any sequence
$(\rho_n)_{n\geq 1}$ with $\limninf \rho_{\negthinspace n} = 0$ the sequence
\begin{equation} \label{regastdot}
   f_n = (\Strho_{\negthinspace n} g_0 \ast \sigma) \cdot \Drho_{\negthinspace n} g_0
\end{equation}
or alternatively
\begin{equation} \label{regdotast}
   \tilde{f}_n = (\Drho_{\negthinspace n} g_0 \cdot \sigma)
   \ast \Strho_{\negthinspace n} g_0
\end{equation}
describe (equivalent) mild Cauchy sequences,
both representing the given $\siSOP$.

In fact, one has (for the case of $(f_n)$) for any $(t,s) \in \TFd$:
\begin{equation}
V_g(f_n)(s,t) = \langle  (\Strho_n g_0 \ast \sigma) \cdot \Drho_n g_0,
M_{-s} T_{t} g \rangle =
\end{equation}
by the  evenness of $g_0$ and hence $\Strho_n g_0$,
and writing  for the right hand  for fixed $(t,s) \in \TFd$
simply
$ f = M_{-s} T_t g \in \SORd$:
$$
= \langle  (\Strho_n g_0 \ast  \sigma) \cdot \Drho_n g_0, f \rangle =
  \langle  (\Strho_n g_0 \ast  \sigma), \Drho_n g_0 \cdot f \rangle =
    \sigma (\Strho_n g_0 \ast [\Drho_n g_0 \cdot f]).
$$
Thus it is finally left to us to verify that the argument
of $\sigma$ in the last expression is in fact convergent
to $f$ in the norm of $\SORdN$.

This can be derived from the fact (for any Segal
algebra, see \cite{re68,re71})
\begin{equation}
\| \Strho_n g_0 \ast f  - f \|_\SOsp \to 0 , \quad \mbox{for} \,\,
   n \to \infty, \,\, \forall f \in \SORd,
\end{equation}
and (by just taking the Fourier transform of this last equation
and replacing $\hatf \in \FT(\SORd) = \SORd$ by $f$ again):
\begin{equation}
\| \Drho_{\negthinspace n} g_0 \cdot f  - f \|_\SOsp \to 0 , \quad \mbox{for} \,\,
   n \to \infty, \,\, \forall f \in \SORd.
\end{equation}
These two estimates can be combined to the required claim
by the following chain of inequalities, using the triangular equation
and the $\Lisp$-boundedness of the Dirac sequence $(\Strho_n g)_{n \geq 1}$.
Given $f \in \SORd$ one has for $n$ large enough:
$$
\sonorm{ \Strho_{\negthinspace n} g_0 \ast [\Drho_{\negthinspace n} g_0 \cdot f] - f
}
\leq
\sonorm{ \Strho_{\negthinspace n} g_0 \ast [\Drho_{\negthinspace n} g_0 \cdot f - f ]}
+
\sonorm{  \Strho_{\negthinspace n} g_0 \ast f - f }
$$
\begin{equation} \label{doubleApp1}
\leq \linorm{\Strho_{\negthinspace n} g_0} \cdot \sonorm{\Drho_{\negthinspace n} g_0
\cdot f - f} + \eps
= (\linorm{g} +1 ) \cdot \eps.
\end{equation}
This establishes that each $\siSOP$ defines an {\it mild Cauchy sequence}.

It is clear from the proof that of course the mild distribution
arising from such a mild CS is in fact the original mild distribution,
since we have
\begin{equation} \label{convback1}
 \lim_{n \to \infty}  V_g (f_n)(t,s) = V_g \sigma(t,s), \quad \forall (t,s) \in
 \Rtst.
 \end{equation}
%
Altogether the above considerations show that it is possible to
view $\SOPRdN$ as a kind of ``natural completion'' of $\CbRdN$,
based on the concept of {\it mild Cauchy sequences}, very much
in the spirit of the sequential approach to distributions, as
featured by Lighthill (\cite{li58}) or as
pointed out in the 
book of Bracewell (\cite{br83}, Chap.5 and Chap.6).

\begin{remark}
Note that the relation $\ScRd \ast (\ScRd \cdot \ScPRd) \subset \ScRd$,
combined with  (\ref{doubleApp1})   implies
that $\ScRd$ is dense in $\SORdN$.
\end{remark}
\section{Connections to Gabor Analysis}
In practice it is not possible to compute $V_g(\sigma)$ 
a continuous functions of $2d$ variables
for uncountably many arguments $(t,s) \in \Rtdst$.
Hence one may ask whether
it is enough to verify boundedness resp.\ pointwise convergence
on a sufficiently large resp.\ dense subset of the TF-plane (phase space).

There is a precise answer to this question, closely connected
to the theory of Gabor frames. Without going into details (and
leaving it to the reader to consult with \cite{gr01}, \cite{fest98},
\cite{fest03} or other more recent sources on Gabor analysis)
let us summarize the most important facts. Although one could
use general lattices  $\Lambda = \AA \ast \Zdst \lhd \TFd$,
for suitable non-singular $2d \times 2d$-matrices $\AA$
we  restrict our attention
to lattices of the form $\Lambda = \abZZd$,  with 
$a,b >0$. We call $a$ the {\it time-step} (or grid constant in the
image domain for $d=2$) and $b$ the {\it frequency-step} (or grid constant
in the wave-number domain for $d=2$).

A {\it (regular) Gabor family} generated by the triple $(g,a,b)$ is a family of the
form $ \{M_{nb} T_{ka} g \suth  n,k \in \Zdst \} $, for some $g \in \LtRd$.
In a short-hand notation we will use the more abstract symbol
$g_\lambda := M_{nb} T_{ka} g$ for the Gabor atom located at
$\lambda = (ka,nb)$ in phase space.
A regular Gabor family is a {\it Gabor frame} if the {\it Gabor frame operator}
\begin{equation} \label{gabframop3}
 S(f):=  S_{(g,a,b)}f := \sumlaLa \langle f, \glam \rangle  \glam
\end{equation}
is a bounded and invertible operator on $\LtRdN$. Then any
$f \in \LtRd$ can be represented as a norm-convergent
double sum in $\LtRdN$, of the form
\begin{equation} \label{Gabrepres1}
 f = \sumlaLa V_\gd{(\lambda)} \glam, 
\end{equation}
where $\gd =  S^{-1}(g) = S_{(g,a,b)}^{-1} g$
is the {\it dual Gabor atom},
providing the minimal norm representation of $f$ using the
Gabor family $\gLam$, with the important estimate
\begin{equation}
 \|V_\gd f |_\Lambda\|_\ltLam \leq C \|f\|_\LtRd, \quad
 \forall f \in \LtRd,
\end{equation}
for some constant depending only on $g$ and $(a,b)$ (in fact, suitable
normalized only on $\gamma_0$, i.e. the density of the lattice $\Lambda
= \abZZd$).

It is one of the most important results of Gabor Analysis
that $g \in \SORd$ not only implies that $S_{g,a,b}$ is a bounded
operator on $\SORdN$, due to the  estimate
$$\|(V_\gd(f)(\lambda))|_\Lambda\|_{\liLam}
\leq C_1 \sonorm{f}, \forall f \in \SORd, $$
but also that $\gd$ belongs $\SORd$, see \cite{grle04,grle06}. 
\begin{lemma} \label{SOgabfram1}
Given non-zero $g \in \SORd$ there exists $\gamma_0 > 0 $  such that for
$ a \leq \gamma_0, b \leq \gamma_0$ one has:
$(g,a,b)$ generates a Gabor frame, with $\gd \in \SORd$.
Hence $V_\gd(\sigma)$ is well defined for $\sigma \in \SOPRd$,
and one has: 
\begin{itemize}
\item $f \in \SORd$ if and only if \,\, $V_g(f)|_\Lambda \in \liLam$;
\item $f \in \SORd$ if and only if  \,\, $V_{\gd}(f)|_\Lambda \in \liLam$;
\item $f \in \SORd$ if and only if \,\,  $f$ has a representation
\begin{equation} \label{Gaborexp4}
f = \sumlaLa c_\lambda \glam, \,\, \mbox{for some
 sequence} (c_\lambda)_\lainLa \in \liLam.
\end{equation}
\end{itemize}
In each case the $\lisp$-norm of the involved coefficients (the
infimum over all possible representations in the case of (\ref{Gaborexp4}))
defines an equivalent norm on $\SORdN$.
\end{lemma}
Based on this observation it is easy to verify by simple modifications
of the proofs given above (the details are left
to the interested reader):
\begin{proposition}
In the situation of Lemma \ref{SOgabfram1}
a sequence $(h_n)_{n\geq 1}$ in $\CbRd$ is   a   mild Cauchy sequence
if and only if
$$ \sup_{n \geq 1} \sup_{k,l \in \Zdst} |V_g (h_n)( ak, bl)| < \infty $$
and
\begin{equation}\label{mildCSdef3}
  \left( V_g(h_n)(ak,kl) \right)_{n \geq 1}
\end{equation}
is a Cauchy-sequence (with respect to $n$) for any $k,l \in \Zdst$.
\end{proposition}

For specific Gabor atoms $g \in \SORd$ (sufficient conditions can be
found in \cite{gr96} or \cite{fewe06}) one  can show that the otherwise
necessary condition $ a \cdto b < 1$ is also a sufficient condition
for creating a Gabor frame. For a long time this was known to be a
valid criterion for the Gauss function (see \cite{se92-1} and \cite{ly92}).
Quite recently this criterion has been extended to the class of
 {\it totally positive functions} 
in \cite{grst13}.


\section{Natural Extension of Operators}

The idea of {\it generalized functions} or {\it distributions} is not so
much to discuss linear functionals, but rather treat objects which are
``more general than functions'' as if they where functions. In fact,
one wants to include in the mathematical discussions objects which
arise naturally in analysis, such
 as Dirac measures, Dirac combs or pure frequencies
$\chi_s(t) = e^{2 \pi i \langle s ,t \rangle }$ which are either not
ordinary (pointwise defined) functions resp.\ not integrable.
Still, one would like to define operations like convolution,
pointwise multiplication or Fourier transforms on these objects
in a way which extends these operations, as known for
ordinary functions, as  ``naturally'' as possible.

Distributions can be shifted, multiplied with decent functions,
they can be dilated (or even rotated for $d > 1$) and one can
take a Fourier transform. It is also possible to define their
support and it behaves in the expected way (e.g.\ with respect
to translation or dilation operators), we will not discuss this
in detail here.

Clearly the extension of the operators defined on ordinary functions,
say at least on $\SORd$ or $\CbRd$ to the more general setting
of $\SOPRd$ or in our setting to $ECmiCS$ should be compatible,
i.e.\ the extended operator should always be the (only) natural
operator defined on the extension of the given operator in such
a way that one has for any ordinary function:

{\bf{Given an ordinary function and applying the operator
first and then the embedding into ECmiCS should be the same as
applying the extended operator to the 
ECmiCS generated by the
ordinary function.}}

In the context of the Banach Gelfand Triple  $\SOGTrRd$ one
would argue that the operators should not only map
the Hilbert space to the Hilbert space, but also the
test functions to test functions, and finally the
dual spaces into each other, not only in a norm-continuous
way, but also in a $\wwst$continuous form.
But we will not make use of this connection.

The answer to this request is for most cases as natural as
it is simple. If an ECmiCS is represented by a mild Cauchy
sequence $(f_n)_{n \geq 1}$ one tries to define
the extension operator $\tilde{T}$ and via
$ (Tf_n)_{n \geq 1}$. Of course one has to verify in
concrete cases that this is well defined, i.e.\ maps
mild Cauchy sequences into mild Cauchy sequences and
preserves equivalence classes.

We will provide a short discussion of the key steps
only for the case of the Fourier transform $T = \FT$.
In this case we reduce the discussion to a sequence
$(f_n)_{n \geq 1}$ with $f_n \in \SORd$ for $n \geq 1$.

Let us first consider the sequence $(\hatf_n)_{n \geq 1}$.
In order to check that it is a mild Cauchy sequence we use
a Gaussian window $g = g_0$, because it has the advantage
of being Fourier invariant.  Then by Plancherel's Theorem
(see \cite{gr01}, formula (3.10)):
\begin{equation}\label{Fourgau1}
  V_{g_0}({\hatf}_n)(t,s) = \langle \hatf_n, M_s T_t g_0 \rangle
      =  \langle f_n,   T_{-s} M_{t} g_0 \rangle
       =   e^{-2 \pi i  t \cdot s }  \langle f_n, M_{t} T_{-s} g_0 \rangle,
\end{equation}
or in short
\begin{equation}\label{Fourgau2}
  \left( V_{g_0} ({\hatf}_n)\right) (t,s)
       =   e^{-2 \pi i   t \cdot s} V_{g_0}(f_n)(-s,t), \quad (t,s) \in
       \TFd.
\end{equation}
The same identities also allow to conclude that the mapping $f \mapsto \hatf$
preserves equivalence classes and that of course this extension is compatible
with the usual definition of an extended Fourier transform (possible
for any Fourier invariant Banach space of test functions):
\begin{equation}\label{Foursigma1}
   \hatsi(f) = \sigma(\hatf), \quad f \in \SORd, \siSOP.
\end{equation}

In fact, we verify that for $(t,s) \in \TFd$ one has:
$$V_{g_0}(\hatf)(t,s) =  \langle \hatf, M_s T_t g_0 \rangle =
 \langle  f, \IFT(M_s T_t g_0) \rangle
$$
with the limit
$ lim_{n \to \infty}  V_{g_0}(\hatf_n)(s,t),$
for any mild CS representing $\sigma$.


\section{Alternative Starting Points}

Although we consider $\CbRd$ as a natural starting point when it
comes to 
 describe the ``largest space of signals''
 for which the STFT is still a bounded functions via
{\it equivalence classes of mild Cauchy sequences} ({\bf ECmiCS})
one may ask whether, by the same form of ``completion'',
other starting points could be chosen, in order to get the
same space, but derive more easily additional properties
(like Fourier invariance). Alternatively, one might ask,
whether certain choices which are closer to applications
(like the use of periodic, discrete signals as point of
departure) yield other objects. Fortunately this will not
be the case. The discussion of these two points makes up
the current section.

The first result in this direction describes a general
observation, making use of the functional analytic setting:
\begin{proposition} \label{SUBSPwstapp}
Given a Banach space $\BspN  \hookrightarrow \SOPRdN$ and $\wst$-dense.
%
Assume   that there is a bounded sequence of operators
$(A_n)_{n \geq 1}$ on $\SOPRdN$ such that for any $n \geq 1$
$A_n$ maps $\SOPRd$ continuously into $\BspN$ 
 \begin{equation}\label{wstappBN}
   f = \wstlim_{n \to \infty} \,  A_n(f), \quad \forall f \in \SORd.
\end{equation}
Then the normed space of  {\bf ECmiCS} arising from $\BspN$
can be naturally identified with  the space of {\bf ECmiCS}
 arising from $\CbRdN$.
\end{proposition}

\begin{proof}
First of all we recall that $\SORd$ is $\wst$-dense in $\SOPRd$
and that for any given $\siSOP$ the sequence
$$
f_n = R_n(\sigma) := \Drho_n g_0 \cdot [\Strho_n g_0 \ast \sigma]
$$
is bounded in $\SOPRdN$ and tends to $\sigma$ in the $\wst$-sense.
The sequence of operators $(R_n)$ is the prototype of a sequence
as alluded in the assumptions of this proposition.

For the proof of our claim let us fix the Banach space $\BspN$
and the sequence $(A_n)_{n \geq 1}$.

Given a mild CS-sequence in $\CbRd$ we want to find a representative
in the same equivalence class consisting of members of $\Bsp$ (instead
of the original choice $\CbRd$).

Given thus $(f_n)_{n \geq 1}$ in $\CbRd$ (viewed as elements
of $\SOPRd$) we may consider $b_n :=(A_n f_n)_{n \geq 1}$.
By the assumptions we have $b_n \in \Bsp$ for each $n \geq 1$. Next we
verify that this is (of course) also a mild CS. Given $g \in \SORd$ we
fix $(t,s) \in \TFd$ and put $h = M_s T_t g$ and watch the behavior of
$$  V_g(b_n)(t,s) = V_g(A_n f_n)(t,s) = A_n f_n (h) $$
for $n \to \infty$. One has
$$ | A_n f_n(h) - f_n(h)| =  |f_n(A_n^*h - h)| \to 0 \quad \mbox{for} \,\,
n \to \infty. $$
This implies at once that the new sequence $(b_n)_{n \geq 1}$ is equivalent
to the original one, as well as (by consequence) it is a mild Cauchy sequence
itself.

By a similar argument 
 one verifies that two equivalent sequence
$(f_n)$ and $({\tilde f}_n)$ give rise to equivalent sequences
$A_n(f_n)$ and $A_n(\tilde{f})$, thus showing that the equivalence classes
are preserved by the replacement.

It remains to be shown that the new version of {\bf ECmiCS} allowing
only the representatives from $\Bsp$ describes an equivalent norm, i.e.\
to verify that the restriction/modification in the set of representatives
does not have an effect on the corresponding infimum's norm. This is
quite plausible because the norm estimate only uses $\supnorm{\Vgsi}$
resp.\ $\SOPsp$-norms.

Concrete estimate rely on the boundedness of the family $(A_n)$
on $\SOPRdN$:
$$ \sup_{n \geq 1} \supnorm{V_g(A_n(f_n))} = \sup_{n \geq 1} \sopnorm{A_n(f_n)}
  \leq  C' \cdot \sup_{n \geq 1} \sopnorm{f_n}.
$$
which implies that the corresponding inf-norms for {\bf{ECmiCS}}  are equivalent.
\end{proof}

\section{References and History}

The sequential approach to distribution theory is not new. Usually
described as a way to handle ``generalized functions'' without
making ({\it explicit!}) use of methods from functional analysis
it offers the possibility to deal with objects which cannot be
treated in the context of ordinary functions or even equivalence
classes of measurable functions (as they are described by the
space $\LpRdN$, for $1 \leq p \leq \infty$).

The most well known book in that direction is certainly the
small booklet of Lighthill (first published in 1958, see
\cite{li58}). But there are other ones
following a similar path, such as Jones (\cite{jo66-1}) or later
Antosik/Mikusinki \cite{anmisi73}. In the 
book by Bracewell (\cite{br83}, Chap.5 and Chap.6) also some comments are made about
the possibility of a sequential approach to distribution theory,
avoiding the theory of topological vector-spaces and duality theory,
ready to be used by engineers.

The approach starts with a definition of {\it good functions}  and
{\it fairly good functions}, and the observation that the Fourier
transform of a good functions is again a good function (later on,
that the Fourier inversion theorem applies in a pointwise
sense). In the usual literature describing the Schwartz theory
of {\it tempered distributions} good functions are called
{\it rapidly decreasing}  or also {\it Schwartz functions},
fairly good functions are just the pointwise multipliers of
the Schwartz space (see also \cite{baor14}). 

The {\it sequential approach} to tempered distributions then
goes on the define {\it regular sequences}  of test functions,
showing that for each regular sequence of test functions
$(f_n)$ in $\ScRd$ the limit
\begin{equation} \label{limregsequ}
 \lim_{n \to \infty} \intRd f_n(x) F(x) dx
\end{equation}
exists, for every $F(x) \in \ScRd$. Translated into the
functional analytic setting a regular sequence is a sequence
of {\it regular distributions} arising from test functions
$f_n \in \ScRd$ via
\begin{equation} \label{regdist002}
 \sigma_n(F) = \intRd F(x) f_n(x) dx,
\end{equation}
which is $\wst$-convergent to some limit.

Going on to call the sequence $(f_n)$ a distribution and
use the symbol $f$ for it is similar to identify an
infinite decimal expression with the sequence of
finite decimal approximations by truncating the
(potentially infinite) sequence at some point and
leaving the rest equal to zero.

In this sense the convention
$$  \intRd f(x) F(x) dx := \lim_{n \to \infty} \intRd f_n(x) F(x) dx $$
is meaningful, but one has to be careful not to confuse the
{\it symbolic expression} (left hand side of the above expression)
with an effective integral. But the situation is not so
far from they use of the symbol $1/\pi$ for the multiplicative
inverse of the irrational number $\pi$, which also is quite
different from the rational number $4/3$ which obviously is
the multiplicative inverse to $3/4$, because in the setting
of rational numbers both symbols have an a priori meaning.

Coming back to the situation described by Lighthill it is clear
that there are many sequences of test functions defining
the same distribution (e.g.\ a Dirac Delta distribution),
hence one has to define a natural equivalence relation.
Altogether the starting point for the sequential approach
to tempered distribution is the definition of such a
distribution as an {\it equivalence class of regular
sequences of test functions}.

This approach is also taken over by Jones (\cite{jo66-1}) and
Sequential Approach. Reprints
Anotsik/Mikusinksi,Sikorski in \cite{anmisi73}, where the
reader can find more details, and of course in \cite{li58}.

In order to compare the two settings let us just recall that it
is long known that $\ScRd \hookrightarrow \SORdN$ (see \cite{po80-1})
it is plausible that {\it mild Cauchy sequences}  of test functions
are also {\it regular sequences}. We just of the mention a small
extra condition, which allows us to formulate the following claim properly.
\newpage
\begin{lemma}
\begin{itemize}
\item For every mild CS $(h_n)$ in $\CbRd$ there exists an equivalent
      sequence $(f_n)$ in $\ScRd$. In other words, every equivalence
      class defining a mild distribution can be constituted with the
      help of mild Cauchy sequences from $\ScRd$.
\item Any mild CS in $\ScRd$ is also a regular sequence. Hence
     any mild distribution (viewed as ECmiCS)  o defines a
    tempered (Lighthill) distribution.
\item A regular sequence of test functions defines a mild distribution
if and only if
\begin{equation} \label{defmildbd1}
\sup_{n \geq 1}  \| V_g(f_n)\|_\infty  < \infty.
\end{equation}
\end{itemize}
\end{lemma}
\begin{proof}
The first claim can be verified by simply applying to a given sequence
$(h_n)$ in $\CbRd$ the usual regularization operators, already used
in the earlier part of this note (i.e. smoothing with a Dirac-like
Gaussian and localization by pointwise multiplication with a dilated
Gaussian, very much like Fourier multipliers). By tuning the parameters
it is easy to verify that one can establish equivalence.

The remaining statements follow from the fact that $\ScRd$ is a
{\it dense} subspace of $\SORdN$. Details are left to the reader.
\end{proof}


\begin{remark}
As a final remark let us note that this article is part of a series of
articles aiming at a demonstration that the Banach Gelfand Triple can
be viewed as a universal tool for the treatment of problem in signal
processing. The connections to classical analysis are described in
\cite{fe19}. An alternative approach showing how to introduce
$\SORdN$ based on Riemann integrals is provided in \cite{feja19}. The
content of \cite{fe19-1} and \cite{fe09} show how it could be used
for the teaching of engineers and physicists.
%
There are of course many references to the classical approach to distribution theory
of L.~Schwartz (originally \cite{sc57}), such as \cite{fr98}, see also \cite{de13-2}.
\end{remark}

\section{Acknowledgement}{*}
The author would like to thank Mads Jakobsen and Anupam Gumber for a careful
reading of the manuscript and various corrections of the original version.
%
%
\bibliographystyle{abbrv}  

\end{document}